\documentclass[12 pt]{amsart}
\usepackage{amsmath,times,epsfig,amssymb,amsbsy,amscd,amsfonts,amstext,graphicx,color,bm,amsthm}
\usepackage[all]{xy}
\usepackage{graphicx}

\usepackage[urlcolor=blue]{hyperref}
\usepackage{tikz,subfigure}

\newcommand{\Q}{\mathbb{Q}}
\newcommand{\PPP}{\mathbb{P}}

\renewcommand{\to}{\longrightarrow}

\newtheorem{Theorem}{Theorem}[section]
\newtheorem{Definition}[Theorem]{Definition}

\newtheorem{Lemma}[Theorem]{Lemma}
\newtheorem{Proposition}[Theorem]{Proposition}
\newtheorem{Corollary}[Theorem]{Corollary}
\newtheorem{Remark}[Theorem]{Remark}
\newtheorem{Example}[Theorem]{Example}
\newtheorem{Conjecture}[Theorem]{Conjecture}

\DeclareMathOperator{\Der}{Der}

\DeclareMathOperator{\pdeg}{pdeg}

\DeclareMathOperator{\codim}{codim}

\def \X(#1){\{x_1,\dots, x_{#1}\}}

\newcommand \lcm  {{\rm lcm}}

\DeclareMathOperator{\LT}{LT}
\DeclareMathOperator{\LM}{LM}
\DeclareMathOperator{\LC}{LC}
\DeclareMathOperator{\rk}{rk}
\DeclareMathOperator{\Mat}{Mat}
\DeclareMathOperator{\diag}{diag}
\DeclareMathOperator{\Supp}{Supp}

\newcommand{\A}{\mathcal{A}}

\newcommand \ideal[1] {\langle #1 \rangle}

\begin{document}

\title{Combinatorially equivalent hyperplane arrangements}

\begin{abstract} 
We study the combinatorics of hyperplane arrangements over arbitrary fields. Specifically, we determine in which situation an arrangement and its reduction modulo a prime number have isomorphic lattices via the use of minimal strong $\sigma$-Gr\"obner bases. Moreover, we prove that the Terao's conjecture over finite fields implies the conjecture over the rationals.
\end{abstract}

\author{Elisa Palezzato}
\address{Elisa Palezzato, Department of Mathematics, Hokkaido University, Kita 10, Nishi 8, Kita-Ku, Sapporo 060-0810, Japan.}
\email{palezzato@math.sci.hokudai.ac.jp}
\author{Michele Torielli}
\address{Michele Torielli, Department of Mathematics, GI-CoRE GSB, Hokkaido University, Kita 10, Nishi 8, Kita-Ku, Sapporo 060-0810, Japan.}
\email{torielli@math.sci.hokudai.ac.jp}


\date{\today}
\maketitle


\section{Introduction}

Let $V$ be a vector space of dimension $l$ over a field $K$. Fix a system of coordinates $(x_1,\dots, x_l)$ of $V^\ast$. 
We denote by $S = S(V^\ast) = K[x_1,\dots, x_l]$ the symmetric algebra of $V^\ast$. 
A hyperplane arrangement $\A = \{H_1, \dots, H_n\}$ is a finite collection of hyperplanes in $V$. For a thorough treatment of the theory of hyperplane arrangements and recent developments, see \cite{orlterao}, \cite{palezzato2019lefschetz}, \cite{guo2018falk} and \cite{palezzato2020localiz}.

The \textit{lattice of intersections} $L(\A)$ is a fundamental combinatorial invariant of an arrangement $\A$. In fact one of the most studied topics
in the theory of arrangements is to identify which topological and algebraic invariants of an arrangement are determined by its lattice of intersections.

To pursue this type of questions, Athanasiadis (\cite{athanasiadis1996characteristic}, \cite{athanasiadis1999extended} and \cite{athanasiadis2004generalized}), inspired by \cite{crapo1970foundations} and \cite{blass1998characteristic}, initiated and systematically applied the ``finite field method'', i.e. the study of the combinatorics of arrangements and their reduction modulo prime numbers. See also \cite{bjorner1997subspace} for related work.  After its introduction, this method has been used by several authors (\cite{kamiya2006arrangements}, \cite{kamiya2008periodicity}, \cite{ardila2007computing} and \cite{palezzato2018free}) to solve similar problems.
The purpose of this paper is to study the combinatorics of arrangements over arbitrary fields and determine in which situation an arrangement and its reduction modulo a prime have isomorphic lattices.

The paper is organized as follows. In Section 2, we recall the basic notions on hyperplane arrangements. In Section 3, we describe how to characterize when two arrangements are combinatorially equivalent. In Section 4, we use the results of Section 3 to describe the primes $p$ for which $\A$ and $\A_p$ are combinatorially equivalent. In Section 5, we show that the knowledge of Terao's conjecture in finite characteristic implies the conjecture over the rationals. In Section 6, we describe a method to compute \textit{good primes} via minimal strong $\sigma$-Gr\"obner bases. In Section 7, we show that computing the \textit{good} and \textit{$(\sigma,l)$-lucky primes} for an arrangement is equivalent to compute all the primes that divide its \textit{$\lcm$-period} (as defined in \cite{kamiya2008periodicity}).

\section{Preliminaries}

Let $K$ be a field. A finite set of affine hyperplanes $\A =\{H_1, \dots, H_n\}$ in $K^l$ is called a \textbf{hyperplane arrangement}. 
For each hyperplane $H_i$ we fix a polynomial $\alpha_i\in S= K[x_1,\dots, x_l]$ such that $H_i = \alpha_i^{-1}(0)$, 
and let $Q(\A)=\prod_{i=1}^n\alpha_i$. An arrangement $\A$ is called \textbf{central} if each $H_i$ contains the origin of $K^l$. 
In this case, each $\alpha_i$ is a linear homogeneous polynomial, and hence $Q(\A)$ is homogeneous of degree $n$. 

Define the \textbf{lattice of intersections} of $\A$ by
$$L(\A)=\{\bigcap_{H\in\mathcal{B}}H \mid \mathcal{B}\subseteq\A\},$$
where if $\mathcal{B}=\emptyset$, we identify $\bigcap_{H\in\mathcal{B}}H$ with $K^l$.
 We endow $L(\A)$ with a partial order defined by $X\le Y$ if and only if $Y\subseteq X$, for all $X,Y\in L(\A)$. 
Note that this is the reverse inclusion. Define a rank function on $L(\A)$ by $\rk(X)=\codim(X)$. 
Moreover, we define $\rk(\A)=\codim(\bigcap_{H\in\mathcal{A}}H)$.
$L(\A)$ plays a fundamental role in the study of hyperplane arrangements, in fact it determines the combinatorics of the arrangement.
Let $$L^k(\A)=\{X\in L(\A)~|~\rk(X)=k\},$$ we call $\A$ \textbf{essential} if $L^l(\A)\ne\emptyset$.

Let $\mu\colon L(\A)\to\mathbb{Z}$ be the \textbf{M\"obius function} of $L(\A)$ defined by
$$\mu(X)=
\begin{cases}
      1 & \text{for } X=K^l,\\
      -\sum_{Y<X}\mu(Y) & \text{if } X>K^l.
\end{cases}$$

The \textbf{characteristic polynomial} of $\A$ is $$\chi(\A,t) = \sum_{X\in L(\A)}\mu(X)t^{\dim(X)}.$$ 

Given $\A =\{H_1, \dots, H_n\}$ an arrangement in $K^l$, the operation of \textbf{coning} allows to transform $\A$ into a central arrangement 
$c\A=\{\tilde{H}_1, \dots, \tilde{H}_{n+1}\}$ in $K^{l+1}$. The hyperplane $\tilde{H}_{n+1}$ corresponds to the hyperplane at infinity $H_\infty$ of $\A$.
Moreover, $\bar{\A}=\{\bar{H}_1, \dots, \bar{H}_{n+1}\}$ denotes the projectivization of $c\A$, which is an arrangement induced by $c\A$ in the projective space $K\PPP^l$.
We will say that $\bar{\A}$ is \textbf{essential} if $\bigcap_{i=1}^{n+1}\bar{H}_i=\emptyset$.

Associated to each hyperplane arrangement $\A$, it can be naturally defined its \textbf{Tutte polynomial}
$$T_\A(x,y)=\sum_{\substack{\mathcal{B}\subseteq\A \\ \mathcal{B} \text{ central }}}(x-1)^{\rk(\A)-\rk(\mathcal{B})}(y-1)^{|\mathcal{B}|-\rk(\mathcal{B})}.$$
As shown in \cite{ardila2007computing}, it turns out that the Tutte polynomial and the characteristic polynomial are related by
$$\chi(\A,t)=(-1)^{\rk(\A)}t^{l-\rk(\A)}T_{\A}(1-t,0).$$
It is sometimes useful to consider a simple transformation of the Tutte polynomial.
The \textbf{coboundary polynomial} of $\A$ is
$$\overline{\chi}_\A(x,y)=\sum_{\substack{\mathcal{B}\subseteq\A \\ \mathcal{B} \text{ central }}}x^{\rk(\A)-\rk(\mathcal{B})}(y-1)^{|\mathcal{B}|}.$$
It is easy to check that
$$\overline{\chi}_\A(x,y)=(y-1)^{\rk(\A)}T_\A\big(\frac{x+y-1}{y-1},y\big),$$
and
$$T_\A(x,y)=\frac{1}{(y-1)^{\rk(\A)}}\overline{\chi}_\A((x-1)(y-1),y).$$

\section{Combinatorial equivalence}
The results in this section are a generalization of certain ones from \cite{terao2002moduli}.
Fix a pair $(l, n)$ with $l\ge1$ and $n\ge0$. Let $A_n(K^l)$ be the set of affine arrangements of $n$ distinct linearly ordered hyperplanes in $K^l$. 
In other words, each element $\A$ of $A_n(K^l)$ is a collection $\A =\{H_1, \dots, H_n\}$, where $H_1, \dots, H_n$ are distinct affine hyperplanes in $K^l$.

\begin{Definition} Given $\A\in A_n(K^l)$, define
$$\mathcal{I}(\bar{\A})=\{(i_1,\dots, i_{l+1})\in[n+1]^{l+1}_<~|~\bar{H}_{i_1}\cap\cdots\cap\bar{H}_{i_{l+1}}\ne\emptyset\},$$
where $[n+1]=\{1,\dots, n+1\}$ and $[n+1]^{l+1}_<=\{(i_1,\dots, i_{l+1})\in[n+1]^{l+1}~|~i_1<\cdots< i_{l+1}\}$.
\end{Definition}

The space $\mathcal{I}(\bar{\A})$ allows us to check if $\A$ and $\bar{\A}$ are essential.

\begin{Lemma}\label{lemm:allessentialcheck} Given $\A\in A_n(K^l)$, the following conditions are equivalent
\begin{enumerate}
\item $\A$ is essential.
\item $\bar{\A}$ is essential.
\item $\mathcal{I}(\bar{\A})\ne[n+1]^{l+1}_<$.  
\end{enumerate}
\end{Lemma}
\begin{proof} We start by proving that (3) is equivalent to (2). 
If (3) is satisfied, then there exists $(i_1,\dots, i_{l+1})\in[n+1]^{l+1}_<$ such that $\bar{H}_{i_1}\cap\cdots\cap\bar{H}_{i_{l+1}}=\emptyset$, and hence $\bar{\A}$ is essential.
On the other hand, if $\bar{\A}$ is essential then there exist $l+1$ hyperplanes $\bar{H}_{i_1},\dots\bar{H}_{i_{l+1}}$ in $\bar{\A}$ whose intersection is empty. This shows that the conditions (2) and (3) are equivalent.

We will now prove that (1) is equivalent to (3). Condition (3) is equivalent to the existence of $(i_1,\dots, i_{l+1})\in[n+1]^{l+1}_<$ such that $\bar{H}_{i_1}\cap\cdots\cap\bar{H}_{i_{l+1}}=\emptyset$.
This happens if and only if there exist $l$ hyperplanes $H_{i_1}, \dots, H_{i_{l}}\in\A$ such that $\bar{H}_{i_1}\cap\cdots\cap\bar{H}_{i_{l}}\cap\bar{H}_{n+1}=\emptyset$ if and only if there exist $l$ hyperplanes $H_{i_1}, \dots, H_{i_{l}}\in\A$ such that $H_{i_1}\cap\cdots\cap H_{i_l}$ is a point. This last fact is equivalent to (1).
\end{proof}

Let $K_1$ and $K_2$ be two fields (non necessarily distinct), and consider $\A^{(j)} =\{H^{(j)}_1, \dots, H^{(j)}_n\}\in A_n(K_j^l)$, for $j=1,2$, two hyperplane arrangements.
\begin{Definition} $\A^{(1)}$ and $\A^{(2)}$ are \textbf{combinatorially equivalent} if
$$\dim(H^{(1)}_{i_1}\cap\cdots\cap H^{(1)}_{i_k}) = \dim(H^{(2)}_{i_1}\cap\cdots\cap H^{(2)}_{i_k}),$$
for all $1\le k\le n$ and $1\le i_1<\cdots<i_k\le n$, where the dimension of the empty set is equal to $-1$. In this case, we write $\A^{(1)}\backsim\A^{(2)}$.
\end{Definition}



The following result is a generalization of \cite[Proposition 3]{terao2002moduli}. 

\begin{Theorem}\label{theo:combinatequivresult} Let $\A$ be an essential arrangement in $K^l$. Then $\mathcal{I}(\bar{\A})$ determines $L(\A)$, and vice versa.
\end{Theorem}

\begin{proof} Consider $(i_1, \dots, i_k)\in[n]^k_<$. Since $\A$ is essential, then $\dim(H_{i_1}\cap\cdots\cap H_{i_k})=l-k $ if and only if there exist $1\le i_{k+1}<\cdots<i_l\le n$ such that $\dim(H_{i_1}\cap\cdots\cap H_{i_l})=0$. Passing to the projectivization, this is equivalent to
the existence of $1\le i_{k+1}<\cdots<i_l\le n$ such that $\bar{H}_{i_1}\cap\cdots\cap \bar{H}_{i_l}\cap\bar{H}_\infty=\emptyset$. This fact is then equivalent to the existence of $1\le i_{k+1}<\cdots<i_l\le n$ such that $(i_1,\dots, i_l, n+1)\notin\mathcal{I}(\bar{\A})$. From the knowledge of which $(i_1, \dots, i_k)\in[n]^k_<$ have  $\dim(H_{i_1}\cap\cdots\cap H_{i_k})=l-k$, we can easily reconstruct $L(\A)$. This shows that $\mathcal{I}(\bar{\A})$ determines $L(\A)$.

Consider $(i_1,\dots, i_{l+1})\in[n+1]^{l+1}_<$. If $i_{l+1} = n+1$, then $\bar{H}_{i_{l+1}}=\bar{H}_\infty$. Moreover, $(i_1,\dots, i_{l+1})\notin\mathcal{I}(\bar{\A})\Leftrightarrow \bar{H}_{i_1}\cap\cdots\cap\bar{H}_{i_l}\cap\bar{H}_\infty=\emptyset$
$\Leftrightarrow H_{i_1}\cap\cdots\cap H_{i_l}$ is a point  $\Leftrightarrow \dim(H_{i_1}\cap\cdots\cap H_{i_l})=0.$
Suppose now that $i_{l+1} < n+1$ and let $\mathcal{B}=\{H_{i_1}, \dots, H_{i_{l+1}}\}$. We have $(i_1,\dots, i_{l+1})\notin\mathcal{I}(\bar{\A})\Leftrightarrow$  $\bar{H}_{i_1}\cap\cdots\cap\bar{H}_{i_{l+1}}=\emptyset$
$\Leftrightarrow H_{i_1}\cap\cdots\cap H_{i_{l+1}}=\emptyset$ and $\bar{H}_{i_1}\cap\cdots\cap\bar{H}_{i_{l+1}}\cap\bar{H}_\infty=\emptyset$ 
$\Leftrightarrow H_{i_1}\cap\cdots\cap H_{i_{l+1}}=\emptyset$ and $\bar{\mathcal{B}}$ is essential. By Lemma~\ref{lemm:allessentialcheck}, this is equivalent to $H_{i_1}\cap\cdots\cap H_{i_{l+1}}=\emptyset$ and $\mathcal{B}$ is essential. This fact is then equivalent to $\dim(H_{i_1}\cap\cdots\cap H_{i_{l+1}})=-1$ and there exist $l$ hyperplanes in $\mathcal{B}$ whose intersection is a point and hence it is zero dimensional. This shows that $L(\A)$ determines $\mathcal{I}(\bar{\A})$.
\end{proof}

\section{Modular case}

From now on we will assume that $\A=\{H_1,\dots, H_n\}$ is a central and essential arrangement in $\mathbb{Q}^l$. After clearing denominators, we can suppose that $\alpha_i \in\mathbb{Z}[x_1,\dots, x_l]$ for all $i=1,\dots, n$, and hence that $Q(\A)=\prod_{i=1}^n\alpha_i\in\mathbb{Z}[x_1,\dots, x_l]$. Moreover, we can also assume that there exists no prime number $p$ that divides any $\alpha_i$.

Let $p$ be a prime number, and consider the canonical homomorphism $$\pi_p\colon \mathbb{Z}[x_1,\dots, x_l]\to \mathbb{F}_p[x_1,\dots, x_l].$$ Since $\A$ is central and we assume that there exists no prime number $p$ that divides any $\alpha_i$, this implies that $\pi_p(\alpha_i)$ is a non-zero linear homogeneous polynomial, for all $i=1,\dots,n$. Since we are interested in the case when $\A$ and its reduction modulo $p$ are both arrangements with the same number of hyperplanes, we call $p$ \textbf{good} for $\A$ if $\pi_p(Q(\A))$ is reduced. Clearly, this is equivalent to the requirement that $\pi_p(\alpha_i)$ and $\pi_p(\alpha_j)$ are not one multiple of the other, for all $1\le i<j\le n$.
Notice that the number of primes $p$ that are non-good for $\A$ is finite, see \cite{palezzato2018free}.

Let now $p$ be a good prime for $\A$. Consider $\A_p=\{(H_1)_p,\dots, (H_n)_p\}$ the arrangement in $\mathbb{F}_p^l$ defined by $\pi_p(Q(\A))\in\mathbb{F}_p[x_1,\dots, x_l]$ and define $(\alpha_i)_p=\pi_p(\alpha_i)$. Hence, by construction, $\A\in A_n(\mathbb{Q}^l)$ and $\A_p\in A_n(\mathbb{F}^l_p)$. Moreover, since $\A$ is central, also $\A_p$ is central.

\begin{Definition}
Given $\A=\{H_1,\dots, H_n\}\in A_n(K^l)$, define $$\mathfrak{I}(\A)=\{(i_1,\dots,i_l)\in[n]^l_<~|~\dim(H_{i_1}\cap\cdots\cap H_{i_l})=0\}.$$
\end{Definition}

\begin{Remark}\label{remark:essiffnotempty} $\A$ is essential if and only if $\mathfrak{I}(\A)\ne\emptyset$. \end{Remark}

\begin{Lemma}\label{lemma:towardsamecombin} The following facts are equivalent
\begin{enumerate} 
\item $\mathcal{I}(\bar{\A})=\mathcal{I}(\bar{\A}_p)$.
\item $\mathfrak{I}(\A)=\mathfrak{I}(\A_p)$.
\end{enumerate}
\end{Lemma}
\begin{proof} 
If (1) is satisfied, since $\A$ is essential, then by Lemma~\ref{lemm:allessentialcheck}, also $\A_p$ is essential. 
Similarly, if (2) is satisfied, then by Remark~\ref{remark:essiffnotempty}, also $\A_p$ is essential. 

Since both $\A$ and $\A_p$ are central, then for all  $(i_1,\dots,i_{l+1})\in[n]^{l+1}_<$, we have that $(i_1,\dots,i_{l+1})\in\mathcal{I}(\bar{\A})\cap\mathcal{I}(\bar{\A}_p)$.
Now $(i_1,\dots,i_l)\in\mathfrak{I}(\A)$ if and only if $H_{i_1}\cap\cdots\cap H_{i_l}$ is a point. This is equivalent to $\bar{H}_{i_1}\cap\cdots\cap\bar{H}_{i_l}\cap\bar{H}_\infty=\emptyset$ and hence to $(i_1,\dots,i_l, n+1)\notin\mathcal{I}(\bar{\A})$.
A similar proof shows that $(i_1,\dots,i_l)\in\mathfrak{I}(\A_p)$ if and only if $(i_1,\dots,i_l,n+1)\notin\mathcal{I}(\bar{\A_p})$.
Putting these three properties together we get our result.
\end{proof}

Since the goal of this section is to determine in which situation an arrangement and its reduction modulo a prime number have isomorphic lattices via the use of minimal strong $\sigma$-Gr\"obner bases, we will now recall some properties of ideals in $\mathbb{Z}[x_1,\dots ,x_l]$.

Let $I$ be an ideal of $\mathbb{Z}[x_1,\dots ,x_l]$ and $\sigma$ a term ordering. Given $f\in \mathbb{Z}[x_1,\dots ,x_l]$, we define the \textbf{leading term} of $f$ as $\LT_\sigma(f)=\max_\sigma\{t\in\Supp(f)\}$,
the \textbf{leading coefficient} of $f$ as the coefficient multiplying the $\LT_\sigma(f)$ in the writing of $f$ and we denote it by $\LC_\sigma(f)$, and
the \textbf{leading monomial} of $f$ as $\LM_\sigma(f)=\LC_\sigma(f)\LT_\sigma(f)$.

\begin{Definition} Let $I$ be an ideal of $\mathbb{Z}[x_1,\dots ,x_l]$, $\sigma$ a term ordering and $G=\{g_1,\dots, g_t\}$ a set of non-zero polynomials in $I$.
We say that $G$ is a \textbf{minimal strong $\sigma$-Gr\"obner basis} for $I$ if the following conditions hold true
\begin{enumerate}
\item $G$ forms a set of generators of $I$;
\item for each $f\in I$, there exists $i\in\{1,\dots,t\}$ such that $\LM_\sigma(g_i)$ divides $\LM_\sigma(f)$;
\item if $i\ne j$, then $\LM_\sigma(g_i)$ does not divide $\LM_\sigma(g_j)$.
\end{enumerate}
\end{Definition}

\begin{Remark}[c.f. \cite{adams1994introduction}, Lemma 4.5.8]\label{rem:minstronggbiggb} 
The reduced $\sigma$-Gr\"obner basis of an ideal $I$ of $\mathbb{Z}[x_1,\dots ,x_l]$ is also a minimal strong $\sigma$-Gr\"obner basis of $I$.
Moreover, every minimal strong $\sigma$-Gr\"obner basis of $I$ is also a $\sigma$-Gr\"obner basis. 
\end{Remark}

\begin{Proposition}[\cite{adams1994introduction}, Exercise 4.5.9] Let $I$ be a non-zero ideal of $\mathbb{Z}[x_1,\dots ,x_l]$ and $\sigma$ a term ordering. 
Then there always exists a minimal strong $\sigma$-Gr\"obner basis of $I$.
\end{Proposition}
\begin{Lemma}[{\cite[Lemma 5.9]{palezzato2018free}}]\label{lem:sameLMandLC} 
Let $I$ be an ideal of $\mathbb{Z}[x_1,\dots,x_l]$, and $\sigma$ a term ordering. Let $G_1$ and $G_2$ be two minimal strong 
$\sigma$-Gr\"obner bases of $I$. Then $\{\LM_\sigma(g)~|~g\in G_1\} = \{\LM_\sigma(g)~|~g\in G_2\}$. Consequently, we have $|G_1| = |G_2|$ 
and $\{\LC_\sigma(g)~|~g\in G_1\} = \{\LC_\sigma(g)~|~g\in G_2\}$.
\end{Lemma}

\begin{Remark} The previous lemma implies that $\{\LM_\sigma(g)~|~g\in G\}$ generates the monomial ideal $\LM_\sigma(I)$, for $G$ any minimal strong $\sigma$-Gr\"obner basis of $I$.
\end{Remark}

By Lemma \ref{lem:sameLMandLC}, we can introduce the following definition. See \cite{palezzato2018free} and \cite{pauer1992lucky}, for more details.

\begin{Definition} Let $I$ be an ideal of $\mathbb{Z}[x_1,\dots,x_l]$, and $\sigma$ be a term ordering. 
If a prime number $p$ does not divide the leading coefficient of any polynomial 
in a minimal strong $\sigma$-Gr\"obner basis for $I$, then we will say $p$ is \textbf{$\sigma$-lucky} for $I$.
\end{Definition}

In other words, $p$ is $\sigma$-lucky for $I$ if and only if it is a non-zero divisor in $\mathbb{Z}[x_1,\dots,x_l]/\LM_\sigma(I)$.
\begin{Remark}\label{rem:finiteluckyprime} Given $I$ an ideal of $\mathbb{Z}[x_1,\dots,x_l]$ and $\sigma$ a term ordering, since a minimal strong $\sigma$-Gr\"obner basis is finite,
 then the number of primes that are not $\sigma$-lucky for $I$ is finite.
\end{Remark}

Now that we have all the tools to work with minimal strong $\sigma$-Gr\"obner basis, we can use them to study the combinatorics of arrangements.

\begin{Proposition}\label{prop:luckytosamezerodim} Consider $(i_1,\dots,i_l)\in[n]^l_<$ and $p$ a good prime for $\A$ that is $\sigma$-lucky for the ideal $I_\mathbb{Z}=\ideal{\alpha_{i_1},\dots,\alpha_{i_l}}_\mathbb{Z}\subseteq\mathbb{Z}[x_1,\dots,x_l]$. Then the following fact are equivalent
\begin{enumerate}
\item $(i_1,\dots,i_l)\in\mathfrak{I}(\A)$.
\item $(i_1,\dots,i_l)\in\mathfrak{I}(\A_p)$.
\end{enumerate}  
\end{Proposition}
\begin{proof} Consider the ideal $I=\ideal{\alpha_{i_1},\dots,\alpha_{i_l}}_\Q\subseteq \Q[x_1,\dots,x_l]$ and the ideal $I_p=\ideal{(\alpha_{i_1})_p,\dots,(\alpha_{i_l})_p}\subseteq\mathbb{F}_p[x_1,\dots,x_l]$.

If $(i_1,\dots,i_l)\in\mathfrak{I}(\A_p)$, then $(H_{i_1})_p\cap\cdots\cap (H_{i_l})_p$ is the origin, and hence $I_p=\ideal{x_1,\dots, x_l}$. This implies that for each $i=1,\dots, l$, there exists $f_i\in \mathbb{Z}[x_1,\dots,x_l]$ of degree $1$ such that $x_i+pf_i\in I$. Since $I$ is an ideal in $\Q[x_1,\dots,x_l]$, we can transform the $f_i$ in such way that $f_i\in\Q[x_{i+1},\dots,x_l]$. This gives us that $\ideal{x_1,\dots, x_l}\subseteq I$.
Since $\A$ is central, then $I$ is a homogenous ideal such that $I\subsetneq \Q[x_1,\dots,x_l]$. This shows that $\ideal{x_1,\dots, x_l}=I$ and hence $(i_1,\dots,i_l)\in\mathfrak{I}(\A)$.

To show the opposite inclusion, assume that $(i_1,\dots,i_l)\in\mathfrak{I}(\A)$. This implies that $H_{i_1}\cap\cdots\cap H_{i_l}$ is the origin, and hence $I$ is zero dimensional and $I=\ideal{x_1,\dots, x_l}$.
Since $I_p$ is a homogenous ideal generated in degree $1$, $I_p\subseteq\ideal{x_1,\dots, x_l}$.
Consider now $\{g_1,\dots, g_l\}$ a minimal strong $\sigma$-Gr\"obner basis for $I_\mathbb{Z}$. Since $I$ is zero-dimensional, then $\{\LM_\sigma(g_1),\dots,\LM_\sigma(g_l)\}=\{\lambda_1x_1,\dots,\lambda_lx_l\}$,
where $\lambda_i\in\mathbb{Z}_{>0}$. 
Since we have $g_j=\sum_{k=1}^lh_{kj}\alpha_{i_k}$, for some $h_{kj}\in\mathbb{Z}[x_1,\dots,x_l]$, then $\pi_p(g_j)\in I_p$. Moreover, since $p$ is $\sigma$-lucky for $I_\mathbb{Z}$, then $\pi_p(g_j)\ne0$ and $\LM_\sigma(\pi_p(g_j))=\pi_p(\LM_\sigma(g_j))\ne0$. This implies that for each $i=1,\dots,l$, there exists $f_i\in I_p$ such that $\LT_\sigma(f_i)=x_i$. This shows that $\ideal{x_1,\dots, x_l}\subseteq I_p$ and hence  $I_p=\ideal{x_1,\dots, x_l}$. This implies that $(i_1,\dots,i_l)\in\mathfrak{I}(\A_p)$.
\end{proof}

As described in Proposition~\ref{prop:luckytosamezerodim}, we are interested in $\sigma$-lucky primes for certain ideals over the integers. This fact motivates the following definition.

\begin{Definition}\label{def:l-lucky} Consider an integer $1\le k\le n$. A prime number $p$ is called \textbf{$(\sigma,k)$-lucky} for $\A$, if it is $\sigma$-lucky for all the ideals of the form $\ideal{\alpha_{i_1},\dots,\alpha_{i_k}}_\mathbb{Z}$, where $\codim(H_{i_1}\cap\cdots\cap H_{i_k})=k$. 
\end{Definition}

\begin{Remark} A prime number $p$ is $(\sigma,l)$-lucky for $\A$, if it is $\sigma$-lucky for all the ideals of the form $\ideal{\alpha_{i_1},\dots,\alpha_{i_l}}_\mathbb{Z}$, for $(i_1,\dots,i_l)\in\mathfrak{I}(\A)$.
\end{Remark}


We can now state the main result of the section.

\begin{Theorem}\label{theo:maintheosamecomb} Let $\A$ be a central and essential arrangement in $\mathbb{Q}^l$. The following facts are equivalent
\begin{enumerate}
\item $p$ is a good and $(\sigma,l)$-lucky prime number for $\A$. 
\item $\A\backsim\A_p$, i.e. $\A$ and $\A_p$ are combinatorially equivalent.
\end{enumerate}
\end{Theorem}
\begin{proof} Assume that $p$ is a good and $(\sigma,l)$-lucky prime number for $\A$. Since $p$ is $(\sigma,l)$-lucky for $\A$, by Proposition~\ref{prop:luckytosamezerodim}, $\mathfrak{I}(\A)=\mathfrak{I}(\A_p)$. By Lemma~\ref{lemma:towardsamecombin}, this implies that $\mathcal{I}(\bar{\A})=\mathcal{I}(\bar{\A}_p)$. We can then conclude that 
$\A\backsim\A_p$ by Theorem~\ref{theo:combinatequivresult}.

Vice versa, assume now that $\A\backsim\A_p$. This clearly implies that $\A$ and $\A_p$ are both (simple)  arrangements with $|\A|=|\A_p|$. This then forces $p$ to be good for $\A$.
Suppose that $p$ is not $(\sigma,l)$-lucky for $\A$. This implies that there exists $\{i_1,\dots,i_l\}\in\mathfrak{I}(\A)$ such that $p$ divides a leading coefficient
in a minimal strong $\sigma$-Gr\"obner basis of $I_\mathbb{Z}=\ideal{\alpha_{i_1},\dots, \alpha_{i_l}}_\mathbb{Z}$. 
Since $\{i_1,\dots,i_l\}\in\mathfrak{I}(\A)$, we can consider $\{g_1,\dots, g_l\}$ a minimal strong $\sigma$-Gr\"obner basis for $I_\mathbb{Z}$ such that $\LM_\sigma(g_i)=\lambda_ix_i$,
where $\lambda_i\in\mathbb{Z}_{>0}$ for all $i=1,\dots, l$. Consider $r=\min\{j\in[l]~|~p \text{ divides } \lambda_j\}$.
Since $\A\backsim\A_p$ and $\{i_1,\dots,i_l\}\in\mathfrak{I}(\A)$, then $\{i_1,\dots,i_l\}\in\mathfrak{I}(\A_p)$ and hence $I_p=\ideal{(\alpha_{i_1})_p,\dots, (\alpha_{i_l})_p}=\ideal{x_1,\dots,x_l}$. In particular, $x_r\in I_p$, and hence there exists $g\in\mathbb{Z}[x_1,\dots, x_l]$ such that $f_r=x_r+pg\in I_\mathbb{Z}$. Since $p$ does not divide $\lambda_i$ with $i<r$, there exist $\gamma_1,\dots,\gamma_{r-1}\in\mathbb{Z}$ such that $\tilde{f}_r=f_r+\sum_{j=1}^{r-1}p\gamma_jg_j\in I_\mathbb{Z}$ with $\LM_\sigma(\tilde{f}_r)=(1+p\beta)x_r$ for some $\beta\in\mathbb{Z}$. Clearly, $p$ does not divide $1+p\beta$ and hence $\lambda_rx_r$ does not divide $\LM_\sigma(\tilde{f}_r)$ but this is impossible since $\{g_1,\dots, g_l\}$ is a minimal strong $\sigma$-Gr\"obner basis for $I_\mathbb{Z}$.
\end{proof}

By the discussion at the beginning of Section 4 and Remark~\ref{rem:finiteluckyprime}, the set of prime numbers that are good and $(\sigma,l)$-lucky  for $\A$ is infinite. This implies that Theorem~\ref{theo:maintheosamecomb} is a generalization of \cite[Proposition 3.11.9]{stanley1998enumerative}, since our result describes explicitly how to compute the prime numbers for which $\A$ and $\A_p$ are not combinatorially equivalent.

Since the characteristic polynomial of an arrangement is determined by its lattice of intersections, we have the following
\begin{Corollary} Let $\A$ be a central and essential arrangement in $\mathbb{Q}^l$, and $p$ a good and $(\sigma,l)$-lucky prime number for $\A$. Then $\chi(\A,t)=\chi(\A_p,t)$.
\end{Corollary}

\begin{Remark} Let $q$ be a power of a prime $p$ and $\A_{\mathbb{F}_q}$ the arrangement in $\mathbb{F}_q^l$ defined by the class of $Q(\A)$ in $\mathbb{F}_q[x_1, \dots, x_l]$.
Then the same argument of Theorem~\ref{theo:maintheosamecomb} shows that if $p$ is good and $(\sigma,l)$-lucky for $\A$, then $\A\backsim\A_{\mathbb{F}_q}$.
\end{Remark}

In \cite{ardila2007computing}, Ardila described a finite field method to compute the coboundary polynomial, and hence the Tutte polynomial, of a given arrangement. His result involved the use of powers of large enough primes to make sure that $\A$ and $\A_{\mathbb{F}_q}$ are combinatorially equivalent. Thanks to Theorem~\ref{theo:maintheosamecomb}, we can rewrite his result as follows.

\begin{Theorem} Let $\A$ be a central and essential arrangement in $\mathbb{Q}^l$, and $p$ a good and $(\sigma,l)$-lucky prime number for $\A$. Then
$$\overline{\chi}_\A(q,t)=\sum_{P\in\mathbb{F}_q^l}t^{h(P)},$$
where $h(P)$ denotes the number of hyperplanes of $\A_{\mathbb{F}_q}$ that contain $P$.
\end{Theorem}

\section{On Terao's conjecture}

We first recall the basic notions and properties of free hyperplane arrangements.
 
We denote by $\Der_{K^l} =\{\sum_{i=1}^l f_i\partial_{x_i}~|~f_i\in S\}$ the $S$-module of \textbf{polynomial vector fields} on $K^l$ (or $S$-derivations). 
Let $\delta =  \sum_{i=1}^l f_i\partial_{x_i}\in \Der_{K^l}$. Then $\delta$ is  said to be \textbf{homogeneous of polynomial degree} $d$ if $f_1, \dots, f_l$ are homogeneous polynomials of degree~$d$ in $S$. 
In this case, we write $\pdeg(\delta) = d$.

Let $\A$ be a central arrangement in $K^l$. Define the \textbf{module of vector fields logarithmic tangent} to $\A$ (or logarithmic vector fields) by
$$D(\A) = \{\delta\in \Der_{K^l}~|~ \delta(\alpha_i) \in \ideal{\alpha_i} S, \forall i\}.$$

The module $D(\A)$ is obviously a graded $S$-module and we have $$D(\A)= \{\delta\in \Der_{K^l}~|~ \delta(Q(\A)) \in \ideal{Q(\A)} S\}.$$ 

\begin{Definition} 
A central arrangement $\A$ in $K^l$ is said to be \textbf{free with exponents $(e_1,\dots,e_l)$} 
if and only if $D(\A)$ is a free $S$-module and there exists a basis $\delta_1,\dots,\delta_l$ of $D(\A)$ 
such that $\pdeg(\delta_i) = e_i$, or equivalently $D(\A)\cong\bigoplus_{i=1}^lS(-e_i)$.
\end{Definition}

A lot it is known about free arrangements, however there is still some mystery around the notion of freeness. See \cite{orlterao}, \cite{yoshinaga2014freeness}, \cite{Gin-freearr} and  \cite{suyama2019} for more details on freeness. For example, Terao's conjecture asserting the dependence of freeness only on the combinatorics is the longstanding open problem in this area. 

\begin{Conjecture}[Terao] 
The freeness of a hyperplane arrangement depends only on its lattice of intersections.
\end{Conjecture}

In \cite{palezzato2018free}, we characterized the prime numbers $p$ for which the freeness of $\A$ implies the freeness of $\A_p$ and, vice versa, the ones for which the freeness of $\A_p$ implies the freeness of $\A$. Specifically, we proved the following two results.

\begin{Theorem}[{\cite[Theorem 4.3]{palezzato2018free}}]\label{theo:fromchar0tocharP} If $\A$ is a free arrangement in $\mathbb{Q}^l$ with exponents $(e_1,\dots, e_l)$, then $\A_p$ is free in $\mathbb{F}_p^l$ with exponents $(e_1,\dots, e_l)$, 
for all good primes except possibly a finite number of them.
\end{Theorem}

\begin{Theorem}[{\cite[Theorem 6.1]{palezzato2018free}}]\label{theo:fromcharPto0} 
Let $p$ be a good prime number for $\A$ that is $\sigma$-lucky for $J(\A)_\mathbb{Z}$, for some term ordering $\sigma$, where $J(\A)_\mathbb{Z}$ denotes the Jacobian ideal of $\A$ as ideal of $\mathbb{Z}[x_1,\dots, x_l]$. If $\A_p$ is free in $\mathbb{F}_p^l$ with exponents $(e_1,\dots, e_l)$, 
then $\A$ is free in $\mathbb{Q}^l$ with exponents $(e_1,\dots, e_l)$.
\end{Theorem} 

Putting together Theorems~\ref{theo:maintheosamecomb}, \ref {theo:fromchar0tocharP} and \ref {theo:fromcharPto0}, we can now show that the knowledge of Terao's conjecture in finite characteristic implies the conjecture over the rationals.

\begin{Theorem} 
If Terao's conjecture is true over all $\mathbb{F}_p$, then it is true over $\Q$.
\end{Theorem}
\begin{proof} Let $\A^{(1)}$ and $\A^{(2)}$ be two central arrangements in $\mathbb{Q}^l$ such that $\A^{(1)}\backsim\A^{(2)}$, and assume that $\A^{(1)}$ is free with exponents $(e_1,\dots, e_l)$.

Consider $\mathcal{P}$ the set of prime numbers that are good and $(\sigma,l)$-lucky for $\A^{(1)}$ and $\A^{(2)}$, and that are $\sigma$-lucky for $J(\A^{(2)})_\mathbb{Z}$. By the discussion at the beginning of Section 4 and Remark~\ref{rem:finiteluckyprime}, $\mathcal{P}$ is infinite.
For every $p\in\mathcal{P}$, Theorem~\ref{theo:maintheosamecomb} gives us  $(\A^{(1)})_p\backsim\A^{(1)}\backsim\A^{(2)}\backsim(\A^{(2)})_p$. On the other hand, by Theorem~\ref{theo:fromchar0tocharP}, we can chose $p\in\mathcal{P}$ in such way that $(\A^{(1)})_p$ is free with exponents $(e_1,\dots, e_l)$. If Terao's conjecture is true over $\mathbb{F}_p$, then $(\A^{(2)})_p$ is free with exponents $(e_1,\dots, e_l)$. Finally by definition of $\mathcal{P}$ and Theorem~\ref{theo:fromcharPto0}, $\A^{(2)}$ is free with exponents $(e_1,\dots, e_l)$.
\end{proof}

It is a natural question to ask if, under the hypothesis of Theorem~\ref{theo:fromcharPto0}, $\A$ and $\A_p$ are combinatorially equivalent. 
In all the examples we considered so far, we obtained a positive answer. This is because in all considered examples, if $p$ is $\sigma$-lucky for $J(\A)_\mathbb{Z}$, 
then it is $(\sigma,l)$-lucky for $\A$. However in general, the converse is not true.

\begin{Example}
Consider the arrangement $\A$ in $\mathbb{Q}^3$ with defining polynomial $Q(\A)=xyz(x+y)(x+2y+z)$. Now $2$ is the only prime that is not $(\sigma,3)$-lucky for $\A$.
On the other hand a direct computation shows that $2$, $3$ and $5$ are not $\sigma$-lucky for $J(\A)_\mathbb{Z}$. 
\end{Example}

\section{How to compute good primes via Gr\"obner bases}

We will now describe a method to compute good primes for an arrangement using minimal strong $\sigma$-Gr\"obner bases.

\begin{Lemma}\label{lemma:nongoognonlucky} Let $1\le i<j \le n$. If $(\alpha_i)_p=\beta(\alpha_j)_p$ for some $\beta\in\mathbb{F}_p\setminus\{0\}$, then $p$ is not  $\sigma$-lucky for the ideal $\ideal{\alpha_i,\alpha_j}_\mathbb{Z}\subseteq\mathbb{Z}[x_1,\dots,x_l].$
\end{Lemma}
\begin{proof} By construction $\alpha_i$ and $\alpha_j$ are distinct homogenous polynomials of degree $1$, that are not one multiple of the other. 
This implies that there exist $g_1, g_2\in\mathbb{Z}[x_1,\dots,x_l]$ two homogenous polynomials of degree $1$ that form a minimal strong $\sigma$-Gr\"obner basis for $\ideal{\alpha_i,\alpha_j}_\mathbb{Z}$. Notice that in this situation $\LM_\sigma(g_k)=\lambda_kx_{i_k}$ for $k=1,2$ with $x_{i_1}\ne x_{i_2}$.

Assume by absurd that $(\alpha_i)_p=\beta(\alpha_j)_p$ for some $\beta\in\mathbb{F}_p\setminus\{0\}$, but $p$ is $\sigma$-lucky for $\ideal{\alpha_i,\alpha_j}_\mathbb{Z}.$
In this situation $\LT_\sigma(\ideal{(\alpha_i)_p, (\alpha_j)_p})=\ideal{x_r}$ for some $1\le r\le l$. 
On the other hand, since $p$ is $\sigma$-lucky for $\ideal{\alpha_i,\alpha_j}_\mathbb{Z}$, we have $\LM_\sigma(\pi_p(g_k))=\pi_p(\LM_\sigma(g_k))=\pi_p(\lambda_k)x_{i_k}\ne0$, for $k=1,2$. This implies that $0\ne\pi_p(g_k)\in\ideal{(\alpha_i)_p, (\alpha_j)_p}$ for $k=1,2$, and hence that $x_{i_1},x_{i_2}\in \LT_\sigma(\ideal{(\alpha_i)_p, (\alpha_j)_p})=\ideal{x_r}$. However this is impossible.
\end{proof}

In general, the converse of Lemma~\ref{lemma:nongoognonlucky} does not hold.
\begin{Example} Consider $\alpha_1=x+y$ and $\alpha_2=x+3y+z$. Then a direct computation shows $\{x+y, 2y+z\}$ is a minimal strong $\sigma$-Gr\"obner basis for the ideal $\ideal{\alpha_1,\alpha_2}_\mathbb{Z}$, and hence $p=2$ is not a $\sigma$-lucky prime. However, $(\alpha_1)_2=x+y$ and $(\alpha_2)_2=x+y+z$ are not one multiple of the other.
\end{Example}

%

We can now show that in order to compute the good primes, it is enough to compute the $(\sigma,2)$-lucky ones.

\begin{Theorem}\label{theo:2luckyisgood} If $p$ is $(\sigma,2)$-lucky for $\A$, then $p$ is good for $\A$.
\end{Theorem}
\begin{proof} By definition, if $p$ is $(\sigma,2)$-lucky for $\A$, then $p$ is $\sigma$-lucky for all the ideals of the form $\ideal{\alpha_i,\alpha_j}_\mathbb{Z}$ for all pairs $1\le i<j \le n$. By Lemma~\ref{lemma:nongoognonlucky}, $(\alpha_i)_p$ and $(\alpha_j)_p$ are not one multiple of the other for all pairs $1\le i<j \le n$. Hence, $\pi_p(Q(\A))$ is reduced.
\end{proof}

In general the statement of Theorem~\ref{theo:2luckyisgood} is not an equivalence.

\begin{Example} Consider the arrangement $\A$ in $\mathbb{Q}^3$ with defining polynomial $Q(\A)=xy(x+y)(x+3y+z)$. Then a direct computation shows that $p=2$ and $p=3$ are not $(\sigma,2)$-lucky for $\A$. However, all prime numbers are good for $\A$.
\end{Example}

\section{On the period of arrangements}

Let  $\A=\{H_1,\dots, H_n\}$ be a central and essential arrangement in $\mathbb{Q}^l$, with $\alpha_i \in\mathbb{Z}[x_1,\dots, x_l]$ for all $i=1,\dots, n$. Moreover, assume that there exists no prime number $p$ that divides any $\alpha_i$. We can associate to $\A$ a $l\times n$ integer matrix
$$C=(c_1,\dots,c_n)\in\Mat_{l\times n}(\mathbb{Z}) $$
consisting of column vectors $c_i=(c_{1i},\dots,c_{li})^T\in\mathbb{Z}^l$, for $i=1,\dots,n$, such that
$$\alpha_i=\sum_{k=1}^lc_{ki}x_k. $$
Similarly, for each non-empty $J=\{i_1,\dots,i_k\}\subseteq[n]$, we consider the $l\times k$ integer matrix
$$C_J=(c_{i_1},\dots,c_{i_k})\in\Mat_{l\times k}(\mathbb{Z}).$$

For each prime number $p$, we can consider $(C)_p$ and $(C_J)_p$ the reductions of $C$ and $C_J$, respectively, modulo $p$.
Notice that $(C)_p$ is the matrix associated to the arrangement $\A_p$.

Since each $C_J$ is an integer matrix, we can consider its Smith normal form. In particular, there exist two unimodular matrices $S_J\in\Mat_{l\times l}(\mathbb{Z})$ and $T_J\in\Mat_{k\times k}(\mathbb{Z})$ such that 
$$S_JC_JT_J=\left ( \begin{array}{cc}
E_J & O \\
O & O \\
\end{array} \right ),$$
where $E_J$ is the diagonal matrix $\diag(e_{J,1},\dots, e_{J,r})$, with $e_{J,1},\dots, e_{J,r}\in\mathbb{Z}_{>0}$, $e_{J,1}|e_{J,2}|\dots|e_{J,r}$ and $r=\rk(C_J)$. 
Denote $e_{J,r}$ simply by $e(J)$, and let the \textbf{$\lcm$-period} of $\A$ be
$$\rho_0=\lcm\{e(J)~|~J\subseteq[n], 1\le|J|\le l\}. $$
In \cite[Theorem 2.4]{kamiya2008periodicity}, the authors proved the following result.

\begin{Theorem} The function $|M(\A_q )|=|\mathbb{Z}^l_q\setminus\bigcup_{H\in\A_q}H|$ is a monic quasi-polynomial in $q\in\mathbb{Z}_{>0}$ of degree $l$ with a period $\rho_0$, where $\A_q$ is the reduction of $\A$ modulo $q$. 
\end{Theorem}

In \cite{kamiya2008periodicity}, the authors also defined
$$q_0=\max_{\emptyset\ne J\subseteq[n]}\min_{S_J}\max\{|u|~|~u \text{ is an entry of } S_JC_J \text{ or } C_J\}$$
and obtained the following result in Corollary 3.3
\begin{Theorem}\label{theo:teraoisolattice} The lattice of intersections $L_q=L(\A_q)$ is periodic in $q>q_0$ with period $\rho_0$. In other words,
$$L_{q+s\rho_0}\simeq L_q,$$
for all $q>q_0$ and $s\in\mathbb{Z}_{\ge0}$.
\end{Theorem}

As noted in \cite{ardila2007computing}, if $p$ is a large prime number, then $\A$ and $\A_p$ are combinatorially equivalent.
Putting together this fact and Theorem~\ref{theo:teraoisolattice}, we get the following result. 

\begin{Corollary}\label{cor:teraoisolattice} Let $p$ be a prime number such that $p>q_0$ and $p$ is coprime with $\rho_0$. Then $\A$ and $\A_p$ are combinatorially equivalent.
\end{Corollary}

The rest of this section is devoted to show that the hypothesis $p>q_0$ is not necessary. Specifically, we will show that computing the good and $(\sigma,l)$-lucky primes for $\A$
is equivalent to computing all the prime numbers that divide $\rho_0$.

\begin{Proposition}\label{prop:nongoodtoperiod} If $p$ is non-good for $\A$, then $p$ divides $\rho_0$.
\end{Proposition}
\begin{proof} Assume $p$ is non-good for $\A$. This implies that there exist a pair of indices $1\le i< j\le n$ such that
\begin{equation}\label{eq:rankdrop}(\alpha_i)_p=\beta(\alpha_j)_p,\end{equation} 
for some $\beta\in\mathbb{F}_p\setminus\{0\}$. Consider now $J=\{i,j\}\subseteq[n]$. Since $\A$ is central, then \eqref{eq:rankdrop} is equivalent to the fact that
$C_J$ has rank $2$ but $(C_J)_p$ has rank $1$. In particular, we have that the Smith normal form of $C_J$ is of the form
$$E_{i,j}=\left ( \begin{array}{c}
E_J  \\
O  \\
\end{array} \right ),$$
where $E_J=\diag(e_1, e(J))$.
By definition, a matrix and its Smith normal form have the same rank. On the other hand the Smith normal form of $(C_J)_p$, up to transforming the elements on the main diagonal to $1$, is $(E_{i,j})_p$ the reduction modulo $p$ of $E_{i,j}$. This implies that $\rk((E_{i,j})_p)=\rk((C_J)_p)=1.$
As a consequence, $p$ divides $e(J)$ and hence $\rho_0$.
\end{proof}

\begin{Proposition}\label{prop:nonlluckytoperiod} If $p$ is not $(\sigma,l)$-lucky for $\A$, then $p$ divides $\rho_0$.
\end{Proposition}
\begin{proof} Let $p$ be a non $(\sigma,l)$-lucky prime number for $\A$. This implies that there exists $J=\{i_1,\dots,i_l\}\in\mathfrak{I}(\A)$ such that $p$ divides a leading coefficient
in a minimal strong $\sigma$-Gr\"obner basis of the ideal $\ideal{\alpha_{i_1},\dots, \alpha_{i_l}}_\mathbb{Z}$. Since $J\in\mathfrak{I}(\A)$, we have that $C_J$ is a $l\times l$ integer matrix of rank $l$. 
Using the same strategy as when computing the Smith normal form of $C_J$, there exists a unimodular $l\times l$ matrix $T_J$ such that $C_JT_J$ is lower triangular. Since $\rk(C_JT_J)=\rk(C_J)=l$, $C_JT_J$ has only non-zero elements on the main diagonal. Seeing that multiplying $C_J$ on the right by $T_J$ is equivalent to perform only column operations on $C_J$, we have that the columns of $C_JT_J$ represent a minimal strong $\sigma$-Gr\"obner basis of $\ideal{\alpha_{i_1},\dots, \alpha_{i_l}}_\mathbb{Z}$. This implies that $p$ divides one of the elements on the main diagonal of $C_JT_J$, and hence its determinant.
On the other hand, by construction, $C_J$ and $C_JT_J$ have the same Smith normal form $E_J$. This implies that the determinants of $C_JT_J$ and of $E_J$ coincide up to a sign. However since $p$ divides the determinant of $C_JT_J$, $p$ divides the determinant of $E_J$ and hence $e(J)$. Finally, by definition of $\rho_0$, this implies that $p$ divides $\rho_0$.
\end{proof}

\begin{Theorem}\label{theo:periodnongoodlucky} Let $p$ be a prime number. Then the following facts are equivalent
\begin{enumerate}
\item $p$ is non-good or not a $(\sigma,l)$-lucky prime number for $\A$.
\item $p$ divides $\rho_0$.
\end{enumerate}
\end{Theorem}
\begin{proof} By Propositions \ref{prop:nongoodtoperiod} and \ref{prop:nonlluckytoperiod}, (1) implies (2).

On the other hand, assume there exists $p$ a prime number that divides the period $\rho_0$, but $p$ is good and $(\sigma,l)$-lucky for $\A$.
This implies that there exists $J=\{i_1,\dots,i_k\}\subseteq[n]$ such that $e(J)$ is divisible by $p$. In particular, since the Smith normal form of $(C_J)_p$, 
up to transforming the elements on the main diagonal to $1$, is the reduction modulo $p$ of the Smith normal form of $C_J$,
this implies that $\rk(C_J)>\rk((C_J)_p)$ and hence that $\dim(H_{i_1}\cap\cdots\cap H_{i_k})>\dim((H_{i_1})_p\cap\cdots\cap (H_{i_k})_p)$. 
However, this implies that $\A$ and $\A_p$ are not combinatorially equivalent, contradicting Theorem~\ref{theo:maintheosamecomb}.
\end{proof}

\begin{Corollary} If $\rho_0$ is a square free integer, then it is the product of all prime numbers that are non-good or not $(\sigma,l)$-lucky for $\A$.
\end{Corollary}

In general, $\rho_0$ is not a square free integer.

\begin{Example} Consider $\A$ the arrangement in $\mathbb{Q}^3$ with defining polynomial $Q(\A)=z(4x+z)(2x+y)(6x+y+3z)(8x+2y+5z)$.
In this situation, $p=2$ is the only non-good prime number for $\A$. Moreover, all prime numbers $p\ne2$ are $(\sigma,l)$-lucky for $\A$.
A direct computation shows that $\rho_0=16$.
\end{Example}

Putting together Theorems \ref{theo:maintheosamecomb} and \ref{theo:periodnongoodlucky}, we obtain the following result that generalizes
Corollary~\ref{cor:teraoisolattice}.


\begin{Corollary}Let $\A$ be a central and essential arrangement in $\mathbb{Q}^l$. The following facts are equivalent
\begin{enumerate}
\item $p$ is coprime with $\rho_0$. 
\item $\A\backsim\A_p$, i.e. $\A$ and $\A_p$ are combinatorially equivalent.
\end{enumerate}
\end{Corollary}

\bigskip
\paragraph{\textbf{Acknowledgements}} The authors would like to thank M. Yoshinaga for many helpful discussions. During the preparation of this article the second author was supported by JSPS Grant-in-Aid for Early-Career Scientists (19K14493). 

\bibliographystyle{plain}

\end{document}